\documentclass[sn-mathphys,Numbered]{sn-jnl}% Math and Physical Sciences Reference Style
\usepackage{graphicx}%
\usepackage{multirow}%
\usepackage{amsmath,amssymb,amsfonts}%
\usepackage{amsthm}%
\usepackage{mathrsfs}%
\usepackage[title]{appendix}%
\usepackage{xcolor}%
\usepackage{textcomp}%
\usepackage{manyfoot}%
\usepackage{booktabs}%
\usepackage{algorithm}%
\usepackage{algorithmicx}%
\usepackage{algpseudocode}%
\usepackage{listings}%
\theoremstyle{thmstyleone}%
\newtheorem{theorem}{Theorem}%  meant for continuous numbers
\theoremstyle{thmstyletwo}%
\newtheorem{example}{Example}%
\newtheorem{lemma}[theorem]{Lemma}%
\theoremstyle{thmstylethree}%
\newtheorem{definition}{Definition}%
\begin{document}

\title[Article Title]{Optimal Filtering for Interacting Particle Systems}

\author[1]{\fnm{Andrey A.} \sur{Dorogovtsev}}

\author[2]{\fnm{Yuecai} \sur{Han}}

\author[1,2]{\fnm{Kateryna } \sur{Hlyniana}}

\author*[2]{\fnm{Yuhang} \sur{Li}}
\equalcont{These authors contributed equally to this work.}

\affil[1]{\orgdiv{Institute of Mathematics}, \orgname{NAS of Ukraine}, \orgaddress{ \postcode{01024}, \city{kyiv}, \country{Ukraine}}}

\affil[2]{\orgdiv{School of Mathematics}, \orgname{Jilin University}, \orgaddress{ \postcode{130012}, \city{Changchun},   \country{China}}}



\abstract{In this paper, we study the optimal filtering problem for a interacting particle system generated by stochastic differential equations with interaction. By using Malliavin calculus, we construct the differential equation of the covariance process and transform the filter problem to an optimal control problem. Finally we give the necessary condition that the coefficient of the optimal filter should satisfy.}




\maketitle

\section{Introduction}
In this paper, we investigate the optimal filtering problem for a family of particles whose evolution is 
described by a stochastic flow model which is generated by a stochastic differential equation (SDE) with interaction. The concept of SDE with interaction is introduced by A.A. Dorogovtsev \cite{dorogovtsev2003stochastic, dorogovtsev2023measure} in the following form
\begin{equation*}
\left\{\begin{array}{l}
d x(u, t)=a\left(x(u, t), \mu_t, t\right) d t+\sum_{k=1}^{\infty} b_k\left(x(u, t), \mu_t, t\right) d w_k(t), \\
x(u, 0)=u, \quad u \in \mathbb{R}^n, \\
\mu_t=\mu_0 \circ x(\cdot, t)^{-1}.
\end{array}\right.
\end{equation*}
This type of SDE  describes the dynamics of a family of particles $\{x(u, \cdot)\}_{ u \in \mathbb{R}^{d}}$ that start from each point of the space $\mathbb{R}^{d}$. The initial mass distribution of the family of particles is denoted by $\mu_{0},$ which is supposed to be a probability measure on the Borel $\sigma-$field in $\mathbb{R}^{d}$. The mass distribution of the family of particles evolves over time, and at moment $t$, it can be written as the push forward of $\mu_{0}$ by the mapping $x(\cdot, t)$. The drift and diffusion coefficients depend on the mass distribution $\mu_{t}$. This gives the possibility to describe the motion of a particle that depends on the positions of all other particles in the system. 

A natural question arises: how to filter the stochastic dynamical systems with interaction when the measurements are linear in the unobservable variables?  
Based on the works of Kalman, \cite{kalman1961new,kalman1961neww}, the classical Kalman filter has been extensively studied and applied in many areas due to the naturally arise of observation outliers, for example, see \cite{durovic1999robust,lawrence1971kalman,masreliez1975approximate,tang2020robot}. For research on theoretical aspects, the stochastic filter is first studied by Kushner \cite{kushner1964dynamical,kushner1967dynamical} and Stratonovich \cite{stratonovich1960conditional}. Fujisaki et al. \cite{Fujisaki1972stochastic} point out that the optimal filter satisfies a non-linear stochastic partial differential equation (SPDE), known as   Kushner–FKK equation. Kallianpur and Striebel \cite{kallianpur1968estimation,kallianpur1969stochastic} establish the representation of the "unnormalized filter". The linear SPDE governing the dynamics of the "unnormalized filter" is investigated in \cite{duncan1967probability,mortensen1966optimal,zakai1969optimal}, which is called Duncan-Mortensen-Zakai
equation or Zakai's equation.

In this paper, we study the following optimal filtering problem. The state system is defined as
\begin{equation*}
    \left\{\begin{array}{l}
dx(u,t)=\left(A(t)x(u,t)+B(t)\Bar{x}_t\right)dt+\sigma(u,t)dW(t),\, t\in [0,T],\\
x(u,0)=u,\quad u\in\mathbf{R}^n,\\
\mu_t=\mu_0\circ x(\cdot,t)^{-1},
\end{array}\right.
\end{equation*}
and the measurement dynamics is given by
\begin{equation*}
    \left\{\begin{array}{l}
dy(u,t)=\left(C(t)x(u,t)+D(t)\Bar{x}_t\right)dt+\gamma(u,t)dV(t),\, t\in [0,T],\\
y(u,0)=u.
\end{array}\right.
\end{equation*}
Here, $\Bar{x}_t=\int_{\mathbf{R}^n}x\mu_t(dx)=\int_{\mathbf{R}^n}x(u,t)\mu_0(du)$. This system describes a signal system $\{x(u,\cdot)\}_{u\in \mathbf{R}^n}$ that contains many signals starting from different positions and being influenced by each other. We can only observe the process $y$, which contains partial information of $x$. Our goal is to construct a family of observable processes $\{z(u,\cdot)\}_{u\in \mathbf{R}^n}$ to estimate the unobservable signals using $y$. 

Because of the interacting term, it is harder to estimate the covariance of the error process. Thanks to Malliavin's calculus, we can deal with this difficulty. Then we use variation methods (used in \cite{ahmed1991quadratic,ahmed2002filtering}) to get the necessary conditions that the optimal filter should satisfy by introducing some processes and deriving their $G\hat{a}teaux$ derivatives. 

The rest of this paper is organized as follows. In Section 2, we introduce the state and measurement dynamics, then state the form of the filter we consider. In Section 3, we deal with the error process and show the dynamic of its covariance so that the filtering problem transforms into a control problem. In Section 4, we give the necessary condition that the optimal control must satisfy, and the filtering problem is solved at the same time.

\section{Dynamics and filtering problem}
Consider the system governed by the following linear SDE
\begin{equation}\label{x}
    \left\{\begin{array}{l}
dx(u,t)=\left(A(t)x(u,t)+B(t)\Bar{x}_t\right)dt+\sigma(u,t)dW(t),\, t\in [0,T],\\
x(u,0)=u,\quad u\in\mathbf{R}^n,\\
\mu_t=\mu_0\circ x(\cdot,t)^{-1}.
\end{array}\right.
\end{equation}
The measurement dynamics is given by:
\begin{equation}\label{y}
    \left\{\begin{array}{l}
dy(u,t)=\left(C(t)x(u,t)+D(t)\Bar{x}_t\right)dt+\gamma(u,t)dV(t),\, t\in [0,T],\\
y(u,0)=u,
\end{array}\right.
\end{equation}
where $\bar{x}_t=\int_{\mathbf{R}^n}x\mu_t(dx)=\int_{\mathbf{R}^n}x(u,t)\mu_0(du)$. The notation $\bar{\phi}_t=\int_{\mathbf{R}^n}\phi(u,t)\mu_0(du)$ will be used frequently in the following, where $\phi$ denotes arbitrary processes. Here, we assume the processes $\{x,y\}$ take values in $\mathbf{R}^n$ and $\mathbf{R}^m$, respectively. $\{W(t),V(t),t\ge 0\}$ are Brownian motions take values in $\mathbf{R}^d$.
Consequently,  the matrices $\{A,B,\sigma,C,D,\gamma\}$ take values in $\mathbf{R}^{n\times n},\mathbf{R}^{n\times n},\mathbf{R}^{n\times d},\mathbf{R}^{m\times n},\mathbf{R}^{m\times n},\mathbf{R}^{m\times d},$ respectively. 

Let $\mathcal{F}^y_t, t\ge 0$, be an increasing family of sub $\sigma$-algebras of the $\mathcal{F}$ induced by the random process $\{y(u,t),u\in \mathcal{R}^n,t\ge 0\}$. The aim is to find a process $z$ such that for all $u\in\mathbf{R}^n,t\ge 0$, $z(u,t)$ is $\mathcal{F}^y_t$-adapted process satisfying
\begin{align*}
\mathbb{E}z(u,t)=\mathbb{E}x(u,t),
\end{align*}
and $\int_{\mathbf{R}^n}\int_0^T\mathbb{E}\|x(u,t)-z(u,t)\|^2dt\mu_0(du)$ is minimized. In other words, we want to find an unbiased minimum variance filter. Our objective here is to derive the best unbiased minimum variance (UMV) linear filter driven by the observed process $y$,  described by the following SDE:
\begin{align}\label{z}
    \left\{\begin{array}{l}
dz(u,t)=\left(H(t)z(u,t)+M(t)\bar{z}_t\right)dt+\Gamma(t)dy(u,t),\,u\in\mathbf{R}^n, t\in[0,T],\\
z(u,0)=u,
\end{array}\right.
\end{align}
where $H,M,\Gamma$ are suitable matrix-valued functions to be determined.

\section{Reformulation of the filtering problem as a control problem}
We introduce the following basic assumptions:
\begin{enumerate}
	\item[(A1)]  There exist positive definite matrices $Q, Q_0\in\mathbf{R}^{d\times d}$ such that
\begin{align*}
\mathbb{E}\left[(W(t),\xi)(W(t),\eta)\right]=(Q\xi,\eta) t,\\
\mathbb{E}\left[(V(t),\xi)(V(t),\eta)\right]=(Q_0\xi,\eta) t
\end{align*}
for all $\xi,\eta\in\mathbf{R}^d, t\in[0,T]$.

\item[(A2)]  The matrices $A,B,C,D$ are locally integrable and $\sigma(u,\cdot)\in L^2\left([0,T],\mathbf{R}^{n\times d}\right)$, $\gamma(u,\cdot)\in L^2\left([0,T],\mathbf{R}^{m\times d}\right), $ for all $ u\in\mathbf{R}^n$.

\item[(A3)] The noises $W$ and $V$ are mutually 
independent.
\end{enumerate}
~\\
Define the error process:
\begin{align*}
e(u,t)=x(u,t)-z(u,t),\quad u\in\mathbf{R}^n,t\in[0,T].
\end{align*}
We use it to describe the variance of the filter. Here, $x$ is the solution of equation (\ref{x}), and $z$ is the solution of equation (\ref{z}) corresponding to any choice of $H,M,\Gamma$. Then the error process $e$ satisfies the following SDE:
\begin{align}\label{de}
    \left\{\begin{array}{l}
de(u,t)=\left[\left(A(t)-\Gamma(t)C(t)\right)e(u,t)+\left(B(t)-\Gamma(t)D(t)\right)\Bar{e}_t \right]dt\\
\quad\qquad\qquad+\left(A(t)-H(t)-\Gamma(t)C(t)\right)z(u,t)dt\\
\quad\qquad\qquad+\left(B(t)-M(t)-\Gamma(t)D(t)\right)\bar{z}_tdt\\
\quad\qquad\qquad+\sigma(u,t)dW(t)-\Gamma(t)\gamma(u,t)dV(t),
\\e(u,0)=0.
\end{array}\right.
\end{align}
Since the estimate is expected to be unbiased, we determine $H,M$ by $\Gamma$ and denote
\begin{align}\label{hm}
H_\Gamma(t)=A(t)-\Gamma(t)C(t),\notag\\
M_\Gamma(t)=B(t)-\Gamma(t)D(t).
\end{align}
We will provide the solution to (\ref{de}). For this, we integrate equation (\ref{de}) with respect to $\mu_0$, and get
\begin{align*}
    \left\{\begin{array}{l}
d\bar{e}_t=[H_\Gamma(t)+M_\Gamma(t)]\bar{e}_tdt+\bar{\sigma}_tdW(t)-\Gamma(t)\bar{\gamma}_tdV(t)
\\\bar{e}_0=0.
\end{array}\right.
\end{align*}
Using the transition operator $\Phi_\Gamma$, we can get the solution of $\bar{e}$ as
\begin{align}\label{eb}
\bar{e}_t=\int_0^t\Phi_\Gamma(t,\theta)\Bar{\sigma}_\theta dW(\theta)-\int_0^t\Phi_\Gamma(t,\theta)\Gamma(\theta)\Bar{\gamma}_\theta dV(\theta),
\end{align}
where $\Phi_\Gamma$ satisfies
\begin{align*}
    \left\{\begin{array}{l}
\frac{\partial}{\partial t}\Phi_\Gamma(t,s)=[H_\Gamma(t)+M_\Gamma(t)
]\Phi_\Gamma(t,s),
\\\Phi_\Gamma(s,s)=I_n.
\end{array}\right.
\end{align*}
Using another transition operator $\Psi_\Gamma$, we get the solution $e(u,t):$ 
\begin{align}\label{ie}
e(u,t)=&\int_0^t\Psi_\Gamma(t,\theta)M_\Gamma(\theta)\bar{e}_\theta d\theta+\int_0^t\Psi_\Gamma(t,\theta)\sigma(u,\theta)dW(\theta)\notag\\&-\int_0^t\Psi_\Gamma(t,\theta)\Gamma(\theta)\gamma(u,\theta)dV(\theta),
\end{align}
where the operator $\Psi_{\Gamma}$ satisfies 
\begin{align}\label{parPsi}
    \left\{\begin{array}{l}
\frac{\partial}{\partial t}\Psi_\Gamma(t,s)=H_\Gamma(t)\Psi_\Gamma(t,s),
\\\Psi_\Gamma(s,s)=I_n.
\end{array}\right.
\end{align}

Define the covariance operator of the error $e$ by
\begin{align}\label{K}
\left(K(u,t)\xi,\eta\right)=\mathbb{E}\left[\left(e(u,t),\xi\right)\left(e(u,t),\eta\right)\right].
\end{align}
To obtain the state equation of $K$, we introduce the definition and a proposition from Malliavin calculus (refer to \cite{hu2005integral,nualart2018introduction}).

\begin{definition}
[Chapter 6.4 of \cite{hu2005integral} or subsection 2.2 of \cite{han2013maximum}]
Let $\{\xi_k\}_{k=1}^{\infty}$ be an orthonormal basis of $L^2([0,\,T])$ such that $\xi_k$, $k=1,2,\cdots$ are smooth functions on $[0,\,T]$. 
We denote $\tilde{\xi}_{j,l}=\int_0^T\xi_j(t)dB_l(t)$, where $j=1,2,...,$ and $l=1,...,m$. Let $\mathcal{P}$ be the set of all polynomials of the standard Brownian motions $B$ over interval $[0,T]$. Namely, $\mathcal{P}$ contains all elements of the form
\begin{align*}
F(\omega)=f(\tilde{\xi}_{j_1,l_1},...,\tilde{\xi}_{j_n,l_n}),
\end{align*}
where $f$ is a polynomial of $n$ variables. The Malliavin derivative $D_s^lF$ is defined by
\begin{equation*}
D_s^lF=\sum_{i=1}^{n}\frac{\partial f}{\partial x_i}\left(\tilde{\xi}_{j_1,l_1},...,\tilde{\xi}_{j_n,l_n}\right)\xi_{j_i}(s)I_{\{j_i=l\}},\quad \forall 0\le s\le T.
\end{equation*}
\end{definition}
The following important proposition (Page 290 of \cite{han2013maximum} or Theorem 3.15 of \cite{nunno2008malliavin}) will be used,
\begin{align}\label{Malliavin}
\mathbb{E}\left(F\int_0^Tg(t)dB_j(t)\right)=\mathbb{E}\int_0^T(D_t^jF)g(t)dt.
\end{align}

Using this method of Malliavian calculus, we can obtain a functional differential equation for the covariance operator of the error process.
In the following we use the notation $A'$ to denote the transpose of a matrix $A$

\begin{lemma}
Let the assumptions (A1)-(A3) hold, then for each $\Gamma\in L^\infty\left([0,T],\mathbf{R}^{n\times m}\right)$, the error covariance $K$ satisfies the following functional differential equation:
\begin{align}\label{dotK}
\dot{K}(u,t)=K_1(u,t)+K_1(u,t)'
\end{align}
with
\begin{align}\label{K1}
K_1(u,t)=&\int_0^t\left[M_\Gamma(t)\Phi_\Gamma(t,s)+H_\Gamma(t)\right]\bar{\sigma}_sQ\bar{\sigma}_s'F_\Gamma(t,s)'ds \notag\\
&+\int_0^t\left[M_\Gamma(t)\Phi_\Gamma(t,s)+H_\Gamma(t)\right]\Gamma(s)\bar{\gamma}_sQ_0\bar{\gamma}_s'\Gamma(s)'F_\Gamma(t,s)'ds\notag\\
&+\int_0^tH_\Gamma(t)\Psi_\Gamma(t,s)\sigma(u,s)Q\sigma(u,s)'\Psi_\Gamma(t,s)'ds\notag\\
&+\int_0^tH_\Gamma(t)\Psi_\Gamma(t,s)\Gamma(s)\gamma(u,s)Q_0\gamma(u,s)'\Gamma(s)'\Psi_\Gamma(t,s)'ds\notag\\
&+\int_0^tH_\Gamma(t)\Psi_\Gamma(t,s)\sigma(u,s)Q\Bar{\sigma}_s'F_\Gamma(t,s)'ds\notag\\
&+\int_0^t\Psi_\Gamma(t,s)\sigma(u,s)Q\Bar{\sigma}_s'\left[M_\Gamma(t)\Phi_\Gamma(t,s)+H_\Gamma(t)F_\Gamma(t,s)\right]'ds\notag\\
&+\int_0^tH_\Gamma(t)\Psi_\Gamma(t,s)\Gamma(s)\gamma(u,s)Q_0\Bar{\sigma}_s'F_\Gamma(t,s)'ds\notag\\
&+\int_0^t\Psi_\Gamma(t,s)\Gamma(s)\gamma(u,s)Q_0\Bar{\sigma}_s'\left[M_\Gamma(t)\Phi_\Gamma(t,s)+H_\Gamma(t)F_\Gamma(t,s)\right]'ds,
\end{align}
and
\begin{align*}
F_\Gamma(t,s)=\int_s^t\Psi_\Gamma(t,\theta)M_\Gamma(\theta)\Phi_\Gamma(\theta,s)d\theta.
\end{align*}
\end{lemma}
\begin{proof}
For $\Gamma\in L^\infty\left([0,T],\mathbf{R}^{n\times m}\right)$,  assumption (A2) implies that $H_\Gamma,$ and $ M_\Gamma$ are locally integrable, and hence the transition operators $\Phi_{\Gamma}$ and $\Psi_{\Gamma}$ are well defined and 
\begin{align*}
\sup\left\{\|\Phi_\Gamma(t,s)\|, 0\le s\le t\le T\right\}<\infty,\\
\sup\left\{\|\Psi_\Gamma(t,s)\|, 0\le s\le t\le T\right\}<\infty.
\end{align*}
Rewrite equation (\ref{ie}) as
\begin{align*}
e(u,t)=I_1(t)+I_2(u,t)+I_3(u,t).
\end{align*}
Firstly, notice that $I_1(t)$ dependends on $W,$ and $ V$. By applying Malliavin calculus, we have
\begin{align*}
D_\theta^W\big[(\bar{e}_t, \xi)\big]=\bar{\sigma}_\theta'\Phi_\Gamma(t,\theta)'\xi\mathbf{1}_{\{\theta\le t\}},
\end{align*}
and
\begin{align*}
D_\theta^V\big[(\bar{e}_t, \xi)\big]=\bar{\gamma}_\theta'\Gamma(\theta)'\Phi_\Gamma(t,\theta)'\xi\mathbf{1}_{\{\theta\le t\}}.
\end{align*}
Thus, by using formula (\ref{Malliavin}) and the assumption (A1), we derive
\begin{align*}
\mathbb{E}&\left[(I_1(t),\xi)(I_2(u,t),\eta)\right]\\=&\int_0^t\eta'\Psi_\Gamma(t,s)\sigma(u,s)QD_s^W\big[(I_1(t),\xi)\big]ds\\
=&\int_0^t\int_0^t\eta'\Psi_\Gamma(t,s)\sigma(u,s)Q\Bar{\sigma}_s'\Phi_\Gamma(t,s)'M_\Gamma(\theta)'\Psi_\Gamma(t,\theta)'\xi \mathbf{1}_{\{s\le \theta\}} d\theta ds.
\end{align*}
Let us define a function
\begin{align}\label{F}
F_\Gamma(t,s)=\int_s^t\Psi_\Gamma(t,\theta)M_\Gamma(\theta)\Phi_\Gamma(\theta,s)d\theta,
\end{align}
then we can continue:
\begin{align*}
\mathbb{E}&\left[(I_1(t),\xi)(I_2(u,t),\eta)\right]\\
=&\int_0^t\eta'\Psi_\Gamma(t,s)\sigma(u,s)Q\Bar{\sigma}_s'F_\Gamma(t,s)'\xi ds\\
=&\int_0^t\Big(\Psi_\Gamma(t,s)\sigma(u,s)Q\Bar{\sigma}_s'F_\Gamma(t,s)'\xi,\eta\Big)ds.
\end{align*}
Note that the function $F_{\Gamma}$  satisfies $F_\Gamma(t,t)=0$ and 
\begin{align}\label{parF}
\frac{\partial}{\partial t}F_\Gamma(t,s)=M_\Gamma(t)\Phi_\Gamma(t,s)+H_\Gamma(t)F_\Gamma(t,s),\quad 0\le s\le t\le T.
\end{align}

Similarly, we find
\begin{align*}
\mathbb{E}\left[(I_1(t),\xi)(I_3(u,t),\eta)\right]=\int_0^t\Big(\Psi_\Gamma(t,s)\Gamma(s)\gamma(u,s)Q_0\Bar{\sigma}_s'F_\Gamma(t,s)'\xi,\eta\Big)ds.
\end{align*}
The last interacting term $\mathbb{E}\left[(I_2(u,t),\xi)(I_3(u,t),\eta)\right]=0$ because of the independence of processes $W$ and $V$.

By It$\hat{\rm o}$'s symmetry, we have
\begin{align*}
\mathbb{E}\left[(I_2(u,t),\xi)(I_2(u,t),\eta)\right]=\int_0^t\Big(\Psi_\Gamma(t,s)\sigma(u,s)Q\sigma(u,s)'\Psi_\Gamma(t,s)\xi,\eta\Big)ds,
\end{align*}
and
\begin{align*}
\mathbb{E}\left[(I_3(u,t),\xi)(I_3(u,t),\eta)\right]=\int_0^t\Big(\Psi_\Gamma(t,s)\Gamma(s)\gamma(u,s)Q_0\gamma(u,s)'\Gamma(s)'\Psi_\Gamma(t,s)\xi,\eta\Big)ds
\end{align*}
for each pair of $\xi, \eta\in\mathbf{R}^n$. Note that we can rewrite $I_1$ using the function $F_{\Gamma}:$
\begin{align*}
I_1(t)=&\int_0^t\Psi_\Gamma(t,s)M_\Gamma(s)\int_0^s\Phi_\Gamma(t,\theta)\Bar{\sigma}_\theta dW(\theta) ds\\
&-\int_0^t\Psi_\Gamma(t,s)M_\Gamma(s)\int_0^s\Phi_\Gamma(t,\theta)\Gamma(\theta)\Bar{\gamma}_\theta dV(\theta) ds\\
=&\int_0^tF_\Gamma(t,s)\bar{\sigma}_sdW(s)-\int_0^tF_\Gamma(t,s)\Gamma(s)\bar{\gamma}_sdV(s).
\end{align*}
From this representation we get
\begin{align*}
\mathbb{E}\left[(I_1(t),\xi)(I_1(t),\eta)\right]=&\int_0^t\Big(F_\Gamma(t,s)\bar{\sigma}_sQ\bar{\sigma}_s'F_\Gamma(t,s)'\xi,\eta\Big)ds\\
&+\int_0^t\Big(F_\Gamma(t,s)\Gamma(s)\bar{\gamma}_sQ_0\bar{\gamma}_s'\Gamma(s)'F_\Gamma(t,s)'\xi,\eta\Big)ds.
\end{align*}
Putting all terms together, we conclude 
\begin{align*}
\left(K(u,t)\xi,\eta\right)=&\mathbb{E}\left[\left(e(u,t),\xi\right)\left(e(u,t),\eta\right)\right]\\
=&\int_0^t\Big(F_\Gamma(t,s)\bar{\sigma}_sQ\bar{\sigma}_s'F_\Gamma(t,s)'\xi,\eta\Big)ds\\
&+\int_0^t\Big(F_\Gamma(t,s)\Gamma(s)\bar{\gamma}_sQ_0\bar{\gamma}_s'\Gamma(s)'F_\Gamma(t,s)'\xi,\eta\Big)ds\\
&+\int_0^t\Big(\Psi_\Gamma(t,s)\sigma(u,s)Q\sigma(u,s)'\Psi_\Gamma(t,s)'\xi,\eta\Big)ds\\
&+\int_0^t\Big(\Psi_\Gamma(t,s)\Gamma(s)\gamma(u,s)Q_0\gamma(u,s)'\Gamma(s)'\Psi_\Gamma(t,s)'\xi,\eta\Big)ds\\
&+\int_0^t\Big(\Psi_\Gamma(t,s)\sigma(u,s)Q\Bar{\sigma}_s'F_\Gamma(t,s)'\xi,\eta\Big)ds\\
&+\int_0^t\Big(F_\Gamma(t,s)\Bar{\sigma}_sQ\sigma(u,s)'\Psi_\Gamma(t,s)'\xi,\eta\Big)ds\\
&+\int_0^t\Big(\Psi_\Gamma(t,s)\Gamma(s)\gamma(u,s)Q_0\Bar{\sigma}_s'F_\Gamma(t,s)'\xi,\eta\Big)ds
\\&+\int_0^t\Big(F_\Gamma(t,s)\Bar{\sigma}_sQ_0\gamma(u,s)'\Gamma(s)'\Psi_\Gamma(t,s)'\xi,\eta\Big)ds
\end{align*}
for any $\xi, \eta \in\mathbf{R}^n$, which gives us the representation of $K$. Now, using the properties (\ref{parPsi}) and (\ref{parF}), we have
\begin{align*}
\dot{K}(u,t)=&\int_0^t\left[M_\Gamma(t)\Phi_\Gamma(t,s)+H_\Gamma(t)\right]\bar{\sigma}_sQ\bar{\sigma}_s'F_\Gamma(t,s)'ds \\
&+\int_0^tF_\Gamma(t,s)\bar{\sigma}_sQ\bar{\sigma}_s'\left[M_\Gamma(t)\Phi_\Gamma(t,s)+H_\Gamma(t)\right]'ds\\
&+\int_0^t\left[M_\Gamma(t)\Phi_\Gamma(t,s)+H_\Gamma(t)\right]\Gamma(s)\bar{\gamma}_sQ_0\bar{\gamma}_s'\Gamma(s)'F_\Gamma(t,s)'ds\\
&+\int_0^tF_\Gamma(t,s)\Gamma(s)\bar{\gamma}_sQ_0\bar{\gamma}_s'\Gamma(s)'\left[M_\Gamma(t)\Phi_\Gamma(t,s)+H_\Gamma(t)\right]'ds\\
&+\int_0^tH_\Gamma(t)\Psi_\Gamma(t,s)\sigma(u,s)Q\sigma(u,s)'\Psi_\Gamma(t,s)'ds\\
&+\int_0^t\Psi_\Gamma(t,s)\sigma(u,s)Q\sigma(u,s)'\Psi_\Gamma(t,s)'H_\Gamma(t)'ds\\
&+\int_0^tH_\Gamma(t)\Psi_\Gamma(t,s)\Gamma(s)\gamma(u,s)Q_0\gamma(u,s)'\Gamma(s)'\Psi_\Gamma(t,s)'ds\\
&+\int_0^t\Psi_\Gamma(t,s)\Gamma(s)\gamma(u,s)Q_0\gamma(u,s)'\Gamma(s)'\Psi_\Gamma(t,s)'H_\Gamma(t)'ds\\
&+\int_0^tH_\Gamma(t)\Psi_\Gamma(t,s)\sigma(u,s)Q\Bar{\sigma}_s'F_\Gamma(t,s)'ds\\
&+\int_0^t\Psi_\Gamma(t,s)\sigma(u,s)Q\Bar{\sigma}_s'\left[M_\Gamma(t)\Phi_\Gamma(t,s)+H_\Gamma(t)F_\Gamma(t,s)\right]'ds\\
&+\int_0^t\left[M_\Gamma(t)\Phi_\Gamma(t,s)+H_\Gamma(t)F_\Gamma(t,s)\right]\Bar{\sigma}_sQ\sigma(u,s)'\Psi_\Gamma(t,s)'ds\\
&+\int_0^tF_\Gamma(t,s)\Bar{\sigma}_sQ\sigma(u,s)'\Psi_\Gamma(t,s)'H_\Gamma(t)'ds\\
&+\int_0^tH_\Gamma(t)\Psi_\Gamma(t,s)\Gamma(s)\gamma(u,s)Q_0\Bar{\sigma}_s'F_\Gamma(t,s)'ds\\
&+\int_0^t\Psi_\Gamma(t,s)\Gamma(s)\gamma(u,s)Q_0\Bar{\sigma}_s'\left[M_\Gamma(t)\Phi_\Gamma(t,s)+H_\Gamma(t)F_\Gamma(t,s)\right]'ds\\
&+\int_0^t\left[M_\Gamma(t)\Phi_\Gamma(t,s)+H_\Gamma(t)F_\Gamma(t,s)\right]\Bar{\sigma}_sQ_0\gamma(u,s)'\Gamma(s)'\Psi_\Gamma(t,s)'ds\\
&+\int_0^tF_\Gamma(t,s)\Bar{\sigma}_sQ_0\gamma(u,s)'\Gamma(s)'\Psi_\Gamma(t,s)'H_\Gamma(t)'ds\\
:=&K_1(u,t)+K_1(u,t)',
\end{align*}
where
\begin{align*}
K_1(u,t)=&\int_0^t\left[M_\Gamma(t)\Phi_\Gamma(t,s)+H_\Gamma(t)\right]\bar{\sigma}_sQ\bar{\sigma}_s'F_\Gamma(t,s)'ds \\
&+\int_0^t\left[M_\Gamma(t)\Phi_\Gamma(t,s)+H_\Gamma(t)\right]\Gamma(s)\bar{\gamma}_sQ_0\bar{\gamma}_s'\Gamma(s)'F_\Gamma(t,s)'ds\\
&+\int_0^tH_\Gamma(t)\Psi_\Gamma(t,s)\sigma(u,s)Q\sigma(u,s)'\Psi_\Gamma(t,s)'ds\\
&+\int_0^tH_\Gamma(t)\Psi_\Gamma(t,s)\Gamma(s)\gamma(u,s)Q_0\gamma(u,s)'\Gamma(s)'\Psi_\Gamma(t,s)'ds\\
&+\int_0^tH_\Gamma(t)\Psi_\Gamma(t,s)\sigma(u,s)Q\Bar{\sigma}_s'F_\Gamma(t,s)'ds\\
&+\int_0^t\Psi_\Gamma(t,s)\sigma(u,s)Q\Bar{\sigma}_s'\left[M_\Gamma(t)\Phi_\Gamma(t,s)+H_\Gamma(t)F_\Gamma(t,s)\right]'ds\\
&+\int_0^tH_\Gamma(t)\Psi_\Gamma(t,s)\Gamma(s)\gamma(u,s)Q_0\Bar{\sigma}_s'F_\Gamma(t,s)'ds\\
&+\int_0^t\Psi_\Gamma(t,s)\Gamma(s)\gamma(u,s)Q_0\Bar{\sigma}_s'\left[M_\Gamma(t)\Phi_\Gamma(t,s)+H_\Gamma(t)F_\Gamma(t,s)\right]'ds.
\end{align*}
\end{proof}

Now, we are prepared to formulate the ﬁltering problem as a control problem. First, we recall that the equation for the linear filter (\ref{z}), with $\Gamma$ to be determined and $H, M$ are determined by $\Gamma$ through (\ref{hm}), gives an unbiased estimate of $x.$ The covariance of the estimate is defined by (\ref{K}) and satisfies the equation (\ref{dotK}). Then we choose $\Gamma$ to obtain a filter with a minimum variance estimate, which minimizes $TrK(u,t)$. In this paper, we consider a more general case. 
Our aim is to minimize $Tr\big(\Sigma(t)K(u,t)\big),$
 where $\Sigma\in L^1\left([0,T],\mathbf{R}^{n\times n}\right)$ is any given real positive deﬁnite symmetric
matrix-valued function. Therefore, the optimal filtering
problem is equivalent to the optimal control problem: ﬁnd $\Gamma\in L^\infty\big([0,T],\mathbf{R}^{n\times m}\big)$ that imparts a minimum to the
functional $J$ subject to the dynamic constraint (\ref{dotK}).

\section{The optimal filter}
We have seen in the preceding section that the optimal filtering problem can be determined by solving the following control problem with interaction: the state equation is given by:
\begin{align}\label{state}
    \left\{\begin{array}{l}
\dot{K}(u,t)=K_1(u,t)+K_1(u,t)',\, t\in[0,T],
\\K(u,0)=0,\qquad u\sim \mu_0,
\end{array}\right.
\end{align}
where $K_1$ is defined as (\ref{K1}). The cost function is defined as
\begin{align}\label{cost}
J(\Gamma)=\int_{\mathbf{R}^n}\int_0^TTr\big(\Sigma(t)K(u,t)\big)dt\mu_0(du).
\end{align}
For simplicity of notations, we will solve the 1-dimensional case in this section. In this case, $Q=Q_0=1$. Moreover, $\Phi, \Psi$ are given by $\Phi_\Gamma(t,s)=exp\{\int_s^t H_\Gamma(r)+M_\Gamma(r)dr\}, \,\Psi_\Gamma(t,s)=exp\{\int_s^t H_\Gamma(r)dr\}$.

We use the variation technique to obtain the necessary conditions of the optimal control and, ultimately, the optimal ﬁlter. So the $G\hat{a}teaux$ derivative of $K$ with respect to $\Gamma$ on $L^\infty\big([0,T], \mathbf{R}\big)$ is needed. Thus, we must derive the $G\hat{a}teaux$ derivative of $\Phi_\Gamma, \Psi_\Gamma$ and $F_\Gamma$. These are stated in the following results.
~\\

\begin{lemma}
Let $\tilde{\Theta}$ denote the $G\hat{a}teaux$ derivative of the map $\Gamma\to \Theta_\Gamma$ at $\Gamma_0$ in the direction $\beta$ for $\Theta=\Phi,\Psi,F$, where $\beta$ denotes $\Gamma-\Gamma_0$ for any $\Gamma\in L^\infty\big([0,T],\mathbf{R}\big)$. Then we have
\begin{align}\label{tildeP}
\Tilde{\Phi}(t,s)=\int_s^t\Phi_1(t,s,\theta)\beta(\theta)d\theta,
\quad
\Tilde{\Psi}(t,s)&=\int_s^t\Psi_1(t,s,\theta)\beta(\theta)d\theta
\end{align}
and
\begin{align*}
\Tilde{F}(t,s)=\int_s^tF_1(t,s,\theta)\beta(\theta)d\theta
\end{align*}
with
\begin{align*}
\Phi_1(t,s,\theta)=&-\big(C(\theta)+D(\theta)\big)\Phi_{\Gamma_0}(t,s),\\
\Psi_1(t,s,\theta)=&-C(\theta)\Psi_{\Gamma_0}(t,s)
\end{align*}
and
\begin{align*}
F_1(t,s,\theta)=&-\int_s^\theta C(\theta)\Psi_{\Gamma_0}(t,r)M_{\Gamma_0}(r)\Phi_{\Gamma_0}(r,s)dr\\
&+\Psi_{\Gamma_0}(t,\theta)D(\theta)\Phi_{\Gamma_0}(\theta,s)\\
&-\int_s^\theta\Psi_{\Gamma_0}(t,r)M_{\Gamma_0}(r)\big(C(\theta)+D(\theta)\big)\Phi_{\Gamma_0}(t,r)dr.
\end{align*}
\end{lemma}

\begin{proof}
Notice that, from straightforward computation, $\tilde{\Phi}$ satisfies equation
\begin{align*}
    \left\{\begin{array}{l}
\frac{\partial}{\partial t}\tilde{\Phi}(t,s)=\big(H_{\Gamma_0}(t)+M_{\Gamma_0}(t)\big)\tilde{\Phi}(t,s)-\beta(t)\big(C(t)+D(t)\big)\Phi_{\Gamma_0}(t,s),
\\\\
\tilde{\Phi}(s,s)=0,
\end{array}\right.
\end{align*}
and similarly,
\begin{align*}
    \left\{\begin{array}{l}
\frac{\partial}{\partial t}\tilde{\Psi}(t,s)=H_{\Gamma_0}(t)\tilde{\Psi}(t,s)-\beta(t)C(t)\Psi_{\Gamma_0}(t,s),
\\\\
\tilde{\Psi}(s,s)=0
\end{array}\right.
\end{align*}
for $0\le s\le t\le T$.
Then it is easy to check $\tilde{\Phi}$ and $\tilde{\Psi}$ satisfy (\ref{tildeP}). To derive the $G\hat{a}teaux$ derivative of $F$  using formulas  (\ref{F}) and (\ref{tildeP}), we get
\begin{align*}
\tilde{F}(t,s)=&\int_s^t\tilde{\Psi}(t,\theta)M_{\Gamma_0}(\theta)\Phi_{\Gamma_0}(\theta,s)d\theta+\int_s^t\Psi_{\Gamma_0}(t,\theta)\tilde{M}(\theta)\Phi_{\Gamma_0}(\theta,s)d\theta\\
&+\int_s^t\Psi_{\Gamma_0}(t,\theta)M_{\Gamma_0}(\theta)\tilde{\Phi}(\theta,s)d\theta\\
=&-\int_s^t\int_\theta^t\beta(r)C(r)\Psi_{\Gamma_0}(t,\theta)drM_{\Gamma_0}(\theta)\Phi_{\Gamma_0}(\theta,s)d\theta\\
&+\int_s^t\Psi_{\Gamma_0}(t,\theta)D(\theta)\beta(\theta)\Phi_{\Gamma_0}(\theta,s)d\theta\\
&-\int_s^t\Psi_{\Gamma_0}(t,\theta)M_{\Gamma_0}(\theta)\int_\theta^t\beta(r)\big(C(r)+D(r)\big)\Phi_{\Gamma_0}(t,\theta)drd\theta\\
:=&\int_s^tF_1(t,s,\theta)\beta(\theta)d\theta,
\end{align*}
where we denoted by
\begin{align*}
F_1(t,s,\theta)=&-\int_s^\theta C(\theta)\Psi_{\Gamma_0}(t,r)M_{\Gamma_0}(r)\Phi_{\Gamma_0}(r,s)dr\\
&+\Psi_{\Gamma_0}(t,\theta)D(\theta)\Phi_{\Gamma_0}(\theta,s)\\
&-\int_s^\theta\Psi_{\Gamma_0}(t,r)M_{\Gamma_0}(r)\big(C(\theta)+D(\theta)\big)\Phi_{\Gamma_0}(t,r)dr.
\end{align*}
\end{proof}
In the next lemma we derive the $G\hat{a}teaux$ derivative of $K$. 

\begin{lemma}
The $G\hat{a}teaux$ derivative of $K$ at $\Gamma_0$ in the direction $\beta$ is given by
\begin{align}\label{tildek}
\tilde{K}(u,t)=\int_0^t K_2(u,t,s)\beta(s)ds,
\end{align}
where
\begin{align*}
&\frac{1}{2}K_2(u,t,s)\\&=\frac{1}{2}\int_0^sF_1(t,\theta,s)F_{\Gamma_0}(t,\theta)\Big(\bar{\sigma}_\theta^2+\Gamma_0(\theta)^2\bar{\gamma}_\theta^2\Big)d\theta\\
&\quad+\int_0^sF_1(t,\theta,s)\bar{\sigma}_\theta
\Psi_{\Gamma_0}(t,\theta)\Big(\sigma(u,\theta)+\Gamma_0(\theta)\gamma(u,\theta)\Big)d\theta\\
&\quad+\int_0^s\Psi_1(t,\theta,s)\Psi_{\Gamma_0}(t,\theta)\Big(\sigma(u,\theta)^2+\Gamma_0(\theta)^2\gamma(u,\theta)^2\Big)d\theta\\
&\quad+\int_0^s\Psi_1(t,\theta,s)F_{\Gamma_0}(t,\theta)\bar{\sigma}_\theta\Big(\sigma(u,\theta)+\Gamma(\theta)\gamma(u,\theta)\Big)d\theta\\
&\quad+\Big(F_{\Gamma_0}(t,s)^2\Gamma_0(s)\Bar{\gamma}_s^2+\Psi_{\Gamma_0}(t,s)^2\Gamma_0(s)\gamma(u,s)^2+\frac{1}{2}\Psi_{\Gamma_0}(t,s)F_{\Gamma_0}(t,s)\bar{\sigma}_s\gamma(u,s)\Big).
\end{align*}
\end{lemma}
\begin{proof}
The proof is similar with Lemma 2. It follows directly from straightforward computation by using the $G\hat{a}teaux$ derivative of $\Phi, \Psi$ and $F$.
\end{proof}
Using the representation of the $G\hat{a}teaux$ derivative of $K$, we can get the  $G\hat{a}teaux$ derivative of the cost function $J$. Namely, we have the following:
\begin{theorem}
The $G\hat{a}teaux$ derivative of $J$ at $\Gamma_0$ in the direction $\beta$ is
\begin{align*}
\tilde{J}=\frac{dJ\Big(\Gamma_0(\cdot)+\varepsilon \beta(\cdot)\Big)}{d\varepsilon}\Bigg|_{\varepsilon=0}=\int_0^T\int_t^T\Sigma(s)\bar{K}_2(s,t)ds\beta(t)dt
\end{align*}
with
\begin{align*}
&\frac{1}{2}\bar{K}_2(t,s)
=\frac{1}{2}\int_{\mathbf{R}}K_2(u,t,s)
\mu_0(du)\\
&=\int_0^sF_1(t,\theta,s)F_{\Gamma_0}(t,\theta)\Big(\bar{\sigma}_\theta^2+\Gamma_0(\theta)^2\bar{\gamma}_\theta^2\Big)d\theta\\
&\quad+\int_0^sF_1(t,\theta,s)\bar{\sigma}_\theta
\Psi_{\Gamma_0}(t,\theta)\Big(\bar{\sigma}_\theta+\Gamma_0(\theta)\bar{\gamma}_\theta\Big)d\theta\\
&\quad+\int_0^s\Psi_1(t,\theta,s)\Psi_{\Gamma_0}(t,\theta)\Big(\overline{\sigma^2}(\theta)+\Gamma_0(\theta)^2\overline{\gamma^2}(\theta)\Big)d\theta\\
&\quad+\int_0^s\Psi_1(t,\theta,s)F_{\Gamma_0}(t,\theta)\bar{\sigma}_\theta\Big(\bar{\sigma}_\theta+\Gamma(\theta)\bar{\gamma}_\theta\Big)d\theta\\
&\quad+\Big(F_{\Gamma_0}(t,s)^2\Gamma_0(s)\Bar{\gamma}_s^2+\Psi_{\Gamma_0}(t,s)^2\Gamma_0(s)\overline{\gamma^2}(s)+\frac{1}{2}\Psi_{\Gamma_0}(t,s)F_{\Gamma_0}(t,s)\bar{\sigma}_s\bar{\gamma}_s\Big).
\end{align*}
\end{theorem}
\begin{proof}
For 1-dimensional case, the cost function has the form
\begin{align*}
J(\Gamma)=\int_{\mathbf{R}}\int_0^T\Sigma(t)K(u,t)dt\mu_0(du).
\end{align*}
Using the $G\hat{a}teaux$ derivative of $K$ (\ref{tildek}), the $G\hat{a}teaux$ derivative is directly given by
\begin{align*}
\Tilde{J}=&\int_{\mathbf{R}}\int_0^T\Sigma(t)\tilde{K}(u,t)dt\mu_0(du)\\
=&\int_{\mathbf{R}}\int_0^T\Sigma(t)\tilde{K}(u,t)dt\mu_0(du)\\
=&\int_{\mathbf{R}}\int_0^T\int_0^t\Sigma(t)K_2(u,t,s)\beta(s)dsdt\mu_0(du)\\
=&\int_0^T\int_t^T\Sigma(s)\bar{K}_2(s,t)ds\beta(t)dt,
\end{align*}
where
\begin{align*}
&\frac{1}{2}\bar{K}_2(t,s)=\frac{1}{2}\int_{\mathbf{R}}K_2(u,t,s)
\mu_0(du)\\
&=\int_0^sF_1(t,\theta,s)F_{\Gamma_0}(t,\theta)\Big(\bar{\sigma}_\theta^2+\Gamma_0(\theta)^2\bar{\gamma}_\theta^2\Big)d\theta\\
&\quad+\int_0^sF_1(t,\theta,s)\bar{\sigma}_\theta
\Psi_{\Gamma_0}(t,\theta)\Big(\bar{\sigma}_\theta+\Gamma_0(\theta)\bar{\gamma}_\theta\Big)d\theta\\
&\quad+\int_0^s\Psi_1(t,\theta,s)\Psi_{\Gamma_0}(t,\theta)\Big(\overline{\sigma^2}(\theta)+\Gamma_0(\theta)^2\overline{\gamma^2}(\theta)\Big)d\theta\\
&\quad+\int_0^s\Psi_1(t,\theta,s)F_{\Gamma_0}(t,\theta)\bar{\sigma}_\theta\Big(\bar{\sigma}_\theta+\Gamma(\theta)\bar{\gamma}_\theta\Big)d\theta\\
&\quad+\Big(F_{\Gamma_0}(t,s)^2\Gamma_0(s)\Bar{\gamma}_s^2+\Psi_{\Gamma_0}(t,s)^2\Gamma_0(s)\overline{\gamma^2}(s)+\frac{1}{2}\Psi_{\Gamma_0}(t,s)F_{\Gamma_0}(t,s)\bar{\sigma}_s\bar{\gamma}_s\Big).
\end{align*}
\end{proof}
Let us assume that $\Gamma_0$ is the optimal control, then we have
\begin{align*}
\tilde{J}=\int_0^T\int_t^T\Sigma(s)\bar{K}_2(s,t)ds\beta(t)dt\ge 0
\end{align*}
for any $\beta\in L^\infty\big([0,T], \mathbf{R}\big)$. Since $\beta$ is arbitrary, from this we get $\int_t^T\Sigma(s)\bar{K}_2(s,t)ds=0$ for all $t\in[0,T].$ This gives us possibility to get the sufficient conditions for the optimal filter:
\begin{theorem}
Assume that $\Gamma_0$ is the optimal control for the control system (\ref{state}), (\ref{cost}), then the following system of equations hold:
\begin{align*}
    \left\{\begin{array}{l}
H_{\Gamma_0}(t)=A(t)-\Gamma_0(t)C(t),\\\\
M_{\Gamma_0}(t)=B(t)-\Gamma_0(t)D(t),\\\\
\Phi_{\Gamma_0}(t,s)=exp\{\int_s^t H_{\Gamma_0}(r)+M_{\Gamma_0}(r)dr\},\\\\
\Psi_{\Gamma_0}(t,s)=exp\{\int_s^t H_{\Gamma_0}(r)dr\},
\\\\
F_{\Gamma_0}(t,s)=\int_s^t\Psi_{\Gamma_0}(t,\theta)M_{\Gamma_0}(\theta)\Phi_{\Gamma_0}(\theta,s)d\theta,\\\\
\Phi_1(t,s,\theta)=-\big(C(\theta)+D(\theta)\big)\Phi_{\Gamma_0}(t,s),\\\\
\Psi_1(t,s,\theta)=-C(\theta)\Psi_{\Gamma_0}(t,s),\\\\
F_1(t,s,\theta)=-\int_s^\theta C(\theta)\Psi_{\Gamma_0}(t,r)M_{\Gamma_0}(r)\Phi_{\Gamma_0}(r,s)dr\\\\
\qquad\qquad\qquad+\Psi_{\Gamma_0}(t,\theta)D(\theta)\Phi_{\Gamma_0}(\theta,s)\\\\
\qquad\qquad\qquad-\int_s^\theta\Psi_{\Gamma_0}(t,r)M_{\Gamma_0}(r)\big(C(\theta)+D(\theta)\big)\Phi_{\Gamma_0}(t,r)dr,\\\\
\frac{1}{2}\bar{K}_2(t,s)\
=\int_0^sF_1(t,\theta,s)F_{\Gamma_0}(t,\theta)\Big(\bar{\sigma}_\theta^2+\Gamma_0(\theta)^2\bar{\gamma}_\theta^2\Big)d\theta\\\\
\qquad\qquad\quad\quad+\int_0^sF_1(t,\theta,s)\bar{\sigma}_\theta
\Psi_{\Gamma_0}(t,\theta)\Big(\bar{\sigma}_\theta+\Gamma_0(\theta)\bar{\gamma}_\theta\Big)d\theta\\\\
\qquad\qquad\quad\quad+\int_0^s\Psi_1(t,\theta,s)\Psi_{\Gamma_0}(t,\theta)\Big(\overline{\sigma^2}(\theta)+\Gamma_0(\theta)^2\overline{\gamma^2}(\theta)\Big)d\theta\\\\
\qquad\qquad\quad\quad+\int_0^s\Psi_1(t,\theta,s)F_{\Gamma_0}(t,\theta)\bar{\sigma}_\theta\Big(\bar{\sigma}_\theta+\Gamma_0(\theta)\bar{\gamma}_\theta\Big)d\theta\\\\
\quad\qquad\qquad\quad+F_{\Gamma_0}(t,s)^2\Gamma_0(s)\Bar{\gamma}_s^2+\Psi_{\Gamma_0}(t,s)^2\Gamma_0(s)\overline{\gamma^2}(s)\\\\\qquad\qquad\quad\quad+\frac{1}{2}\Psi_{\Gamma_0}(t,s)F_{\Gamma_0}(t,s)\bar{\sigma}_s\bar{\gamma}_s,\\\\
\int_t^T\Sigma(s)\bar{K}_2(s,t)ds=0,\quad \forall t\in[0,T].
\end{array}\right.
\end{align*}
At the same time, the optimal filter is given by
\begin{align*}
    \left\{\begin{array}{l}
dz(u,t)=\left(\big(A(t)-\Gamma_0C(t)\big)z(u,t)+\big(B(t)-\Gamma_0D(t)\big)\bar{z}_t\right)dt+\Gamma_0(t)dy(u,t),\\
z(u,0)=u.
\end{array}\right.
\end{align*}
\end{theorem}

\begin{example}
Let $\sigma(u,t)=\gamma(u,t)=u$ and $\mu_0$ be standard normal distribution, namely, $\mu_0(\Delta)=\int_\Delta (2\pi)^{-\frac{1}{2}}e^{-\frac{x^2}{2}}dx$. In this case, $\bar{\sigma}_\theta=\bar{\gamma}_\theta=0$ and $\overline{\sigma^2}_\theta=\overline{\gamma^2}_\theta=1$. So that 
\begin{align*}
\frac{1}{2}\bar{K}_2(t,s)\
=-\int_0^sC(s)\Psi_{\Gamma_0}(t,\theta)^2\Big(1+\Gamma_0(\theta)^2\Big)d\theta
+\Psi_{\Gamma_0}(t,s)^2\Gamma_0(s).
\end{align*}
Let $\bar{K}_2(t,s)=0$ and using $\Psi(t,s)\Psi(s,\theta)=\Psi(t,\theta),\,\forall \theta\le s\le t$, the $\Gamma_0$ is given by
\begin{align*}
\Gamma_0(s)=\int_0^sC(s)\Psi_{\Gamma_0}(s,\theta)^2\Big(1+\Gamma_0(\theta)^2\Big)d\theta,\quad \forall s\in[0,T].
\end{align*}
Let $M(s):=\frac{\Gamma_0(s)}{C(s)}$, through straightforward computation, we have
\begin{align*}
M'(s)=1+C^2(s)M^2(s)+2H_{\Gamma_0}(s)M(s).
\end{align*}
On the other hand, $\bar{K}(t)=\int_{\mathbf{R}} e^2(u,t)\mu_0(du)$ is given by
\begin{align*}
\bar{K}'(t)=1+C^2(t)M^2(t)+2H_{\Gamma_0}(t)\bar{K}(t),
\end{align*}
which shows $\bar{K}(t)=M(t)$.
\end{example}

~\\

\begin{example}
Let $B(t)=D(t)=0$, $\mu_0=\delta_{x_0}$, $\sigma(u,t)=\sigma_0(t), \gamma(u,t)=\gamma_0(t)$. Then the filtering problem transforms to a classical Kalman filtering problem. In this case, $M_\Gamma=F_\Gamma=0$ and 
\begin{align*}
\frac{1}{2}\bar{K}_2(t,s)=&-\int_0^s\Psi_{\Gamma_0}(t,s)C(s)\Psi_{\Gamma_0}(s,\theta)\Psi_{\Gamma_0}(t,\theta)\Big(\sigma_0^2(\theta)+\Gamma_0(\theta)^2\gamma_0^2(\theta)\Big)d\theta\\
&+\Psi_{\Gamma_0}(t,s)^2\Gamma_0(s)\gamma_0^2(s).
\end{align*}
Using the property $\Psi(t,s)\Psi(s,\theta)=\Psi(t,\theta),\,\forall \theta\le s\le t$, let $\Bar{K}_2(t,s)=0$, we conclude
\begin{align*}
\int_0^s \Psi_{\Gamma_0}^2(s,\theta)\left(\sigma_0^2(\theta)+\Gamma_0^2(\theta)\gamma_0^2(\theta)\right)d\theta+\frac{\Gamma_0(s)\gamma_0^2(s)}{C(s)}=0,\,\forall s\in[0,T].
\end{align*}
Let $S(t)=\frac{\gamma_0^2(t)}{C(t)}\Gamma_0(t).$ Thus, we obtain
\begin{align*}
S'(t)=2A(t)S(t)-\frac{C^2(t)}{\gamma_0^2(t)}S^2(t)+\sigma_0^2(t),\quad  S(0)=0
\end{align*}
and the optimal filter is given by
\begin{align*}
dz_0(t)=\left(A(t)-\frac{C^2(t)}{\gamma_0^2(t)}S(t)\right)z_0(t)dt+\frac{C(t)}{\gamma_0^2(t)}S(t)dy_0(t),
\end{align*}
which are similar to the classical results.
\end{example}

\section{Funding and conflict of interest}
This work was supported by National Key R$\&$D Program of China (Grant numbers 2023YFA1009200) and the National Science Foundation of China (Grant numbers 12201241). The author Yuecai Han has received research support from the first funding and Kateryna Hlyniana has received research support from the second funding. There is no conflict of interest.

\bibliography{references}


\begin{thebibliography}{23}
\ifx \bisbn   \undefined \def \bisbn  #1{ISBN #1}\fi
\ifx \binits  \undefined \def \binits#1{#1}\fi
\ifx \bauthor  \undefined \def \bauthor#1{#1}\fi
\ifx \batitle  \undefined \def \batitle#1{#1}\fi
\ifx \bjtitle  \undefined \def \bjtitle#1{#1}\fi
\ifx \bvolume  \undefined \def \bvolume#1{\textbf{#1}}\fi
\ifx \byear  \undefined \def \byear#1{#1}\fi
\ifx \bissue  \undefined \def \bissue#1{#1}\fi
\ifx \bfpage  \undefined \def \bfpage#1{#1}\fi
\ifx \blpage  \undefined \def \blpage #1{#1}\fi
\ifx \burl  \undefined \def \burl#1{\textsf{#1}}\fi
\ifx \doiurl  \undefined \def \doiurl#1{\url{https://doi.org/#1}}\fi
\ifx \betal  \undefined \def \betal{\textit{et al.}}\fi
\ifx \binstitute  \undefined \def \binstitute#1{#1}\fi
\ifx \binstitutionaled  \undefined \def \binstitutionaled#1{#1}\fi
\ifx \bctitle  \undefined \def \bctitle#1{#1}\fi
\ifx \beditor  \undefined \def \beditor#1{#1}\fi
\ifx \bpublisher  \undefined \def \bpublisher#1{#1}\fi
\ifx \bbtitle  \undefined \def \bbtitle#1{#1}\fi
\ifx \bedition  \undefined \def \bedition#1{#1}\fi
\ifx \bseriesno  \undefined \def \bseriesno#1{#1}\fi
\ifx \blocation  \undefined \def \blocation#1{#1}\fi
\ifx \bsertitle  \undefined \def \bsertitle#1{#1}\fi
\ifx \bsnm \undefined \def \bsnm#1{#1}\fi
\ifx \bsuffix \undefined \def \bsuffix#1{#1}\fi
\ifx \bparticle \undefined \def \bparticle#1{#1}\fi
\ifx \barticle \undefined \def \barticle#1{#1}\fi
\bibcommenthead
\ifx \bconfdate \undefined \def \bconfdate #1{#1}\fi
\ifx \botherref \undefined \def \botherref #1{#1}\fi
\ifx \url \undefined \def \url#1{\textsf{#1}}\fi
\ifx \bchapter \undefined \def \bchapter#1{#1}\fi
\ifx \bbook \undefined \def \bbook#1{#1}\fi
\ifx \bcomment \undefined \def \bcomment#1{#1}\fi
\ifx \oauthor \undefined \def \oauthor#1{#1}\fi
\ifx \citeauthoryear \undefined \def \citeauthoryear#1{#1}\fi
\ifx \endbibitem  \undefined \def \endbibitem {}\fi
\ifx \bconflocation  \undefined \def \bconflocation#1{#1}\fi
\ifx \arxivurl  \undefined \def \arxivurl#1{\textsf{#1}}\fi
\csname PreBibitemsHook\endcsname

%%% 1
\bibitem[\protect\citeauthoryear{Dorogovtsev}{2003}]{dorogovtsev2003stochastic}
\begin{barticle}
\bauthor{\bsnm{Dorogovtsev}, \binits{A.A.}}:
\batitle{Stochastic flows with interaction and measure-valued processes}.
\bjtitle{International Journal of Mathematics and Mathematical Sciences}
\bvolume{2003},
\bfpage{3963}--\blpage{3977}
(\byear{2003})
\end{barticle}
\endbibitem

%%% 2
\bibitem[\protect\citeauthoryear{Dorogovtsev}{2023}]{dorogovtsev2023measure}
\begin{botherref}
\oauthor{\bsnm{Dorogovtsev}, \binits{A.A.}}:
Measure-valued processes and stochastic flows.
vol. 3.
Walter de Gruyter GmbH \& Co KG
(2023)
\end{botherref}
\endbibitem

%%% 3
\bibitem[\protect\citeauthoryear{Kalman}{1961}]{kalman1961new}
\begin{botherref}
\oauthor{\bsnm{Kalman}, \binits{R.E.}}:
A new approach to linear filtering and prediction theory.
ASME Journal of Basic Engineering, series D
(1961)
\end{botherref}
\endbibitem

%%% 4
\bibitem[\protect\citeauthoryear{Kalman and Bucy}{1961}]{kalman1961neww}
\begin{barticle}
\bauthor{\bsnm{Kalman}, \binits{R.E.}},
\bauthor{\bsnm{Bucy}, \binits{R.S.}}:
\batitle{New results in linear filtering and prediction theory}.
\bjtitle{Journal of Basic Engineering}
\bvolume{83}(\bissue{1}),
\bfpage{95}--\blpage{108}
(\byear{1961})
\end{barticle}
\endbibitem

%%% 5
\bibitem[\protect\citeauthoryear{Durovic and Kovacevic}{1999}]{durovic1999robust}
\begin{barticle}
\bauthor{\bsnm{Durovic}, \binits{Z.M.}},
\bauthor{\bsnm{Kovacevic}, \binits{B.D.}}:
\batitle{Robust estimation with unknown noise statistics}.
\bjtitle{IEEE Transactions on Automatic Control}
\bvolume{44}(\bissue{6}),
\bfpage{1292}--\blpage{1296}
(\byear{1999})
\end{barticle}
\endbibitem

%%% 6
\bibitem[\protect\citeauthoryear{Lawrence and Kaufman}{1971}]{lawrence1971kalman}
\begin{barticle}
\bauthor{\bsnm{Lawrence}, \binits{R.}},
\bauthor{\bsnm{Kaufman}, \binits{H.}}:
\batitle{The kalman filter for the equalization of a digital communications channel}.
\bjtitle{IEEE Transactions on Communication Technology}
\bvolume{19}(\bissue{6}),
\bfpage{1137}--\blpage{1141}
(\byear{1971})
\end{barticle}
\endbibitem

%%% 7
\bibitem[\protect\citeauthoryear{Masreliez}{1975}]{masreliez1975approximate}
\begin{barticle}
\bauthor{\bsnm{Masreliez}, \binits{C.}}:
\batitle{Approximate non-gaussian filtering with linear state and observation relations}.
\bjtitle{IEEE Transactions on Automatic Control}
\bvolume{20}(\bissue{1}),
\bfpage{107}--\blpage{110}
(\byear{1975})
\end{barticle}
\endbibitem

%%% 8
\bibitem[\protect\citeauthoryear{Tang et~al.}{2020}]{tang2020robot}
\begin{barticle}
\bauthor{\bsnm{Tang}, \binits{M.}},
\bauthor{\bsnm{Chen}, \binits{Z.}},
\bauthor{\bsnm{Yin}, \binits{F.}}:
\batitle{Robot tracking in slam with masreliez-martin unscented kalman filter}.
\bjtitle{International Journal of Control, Automation and Systems}
\bvolume{18}(\bissue{9}),
\bfpage{2315}--\blpage{2325}
(\byear{2020})
\end{barticle}
\endbibitem

%%% 9
\bibitem[\protect\citeauthoryear{Kushner}{1964}]{kushner1964dynamical}
\begin{barticle}
\bauthor{\bsnm{Kushner}, \binits{H.J.}}:
\batitle{On the dynamical equations of conditional probability density functions, with applications to optimal stochastic control theory}.
\bjtitle{Journal of Mathematical Analysis and Applications}
\bvolume{8}(\bissue{2}),
\bfpage{332}--\blpage{344}
(\byear{1964})
\end{barticle}
\endbibitem

%%% 10
\bibitem[\protect\citeauthoryear{Kushner}{1967}]{kushner1967dynamical}
\begin{barticle}
\bauthor{\bsnm{Kushner}, \binits{H.J.}}:
\batitle{Dynamical equations for optimal nonlinear filtering}.
\bjtitle{Journal of Differential Equations}
\bvolume{3}(\bissue{2}),
\bfpage{179}--\blpage{190}
(\byear{1967})
\end{barticle}
\endbibitem

%%% 11
\bibitem[\protect\citeauthoryear{Stratonovich}{1960}]{stratonovich1960conditional}
\begin{barticle}
\bauthor{\bsnm{Stratonovich}, \binits{R.}}:
\batitle{Conditional markov processes}.
\bjtitle{Theory of Probability \& Its Applications}
\bvolume{5}(\bissue{2}),
\bfpage{156}--\blpage{178}
(\byear{1960})
\end{barticle}
\endbibitem

%%% 12
\bibitem[\protect\citeauthoryear{Fujisaki et~al.}{1972}]{Fujisaki1972stochastic}
\begin{barticle}
\bauthor{\bsnm{Fujisaki}, \binits{M.}},
\bauthor{\bsnm{Kallianpur}, \binits{G.}},
\bauthor{\bsnm{Kunita}, \binits{H.}}:
\batitle{Stochastic differential equations for the non linear filtering problem}.
\bjtitle{Osaka Journal of Mathematics}
\bvolume{9}(\bissue{1}),
\bfpage{19}--\blpage{40}
(\byear{1972})
\end{barticle}
\endbibitem

%%% 13
\bibitem[\protect\citeauthoryear{Kallianpur and Striebel}{1968}]{kallianpur1968estimation}
\begin{barticle}
\bauthor{\bsnm{Kallianpur}, \binits{G.}},
\bauthor{\bsnm{Striebel}, \binits{C.}}:
\batitle{Estimation of stochastic systems: Arbitrary system process with additive white noise observation errors}.
\bjtitle{The Annals of Mathematical Statistics}
\bvolume{39}(\bissue{3}),
\bfpage{785}--\blpage{801}
(\byear{1968})
\end{barticle}
\endbibitem

%%% 14
\bibitem[\protect\citeauthoryear{Kallianpur and Striebel}{1969}]{kallianpur1969stochastic}
\begin{barticle}
\bauthor{\bsnm{Kallianpur}, \binits{G.}},
\bauthor{\bsnm{Striebel}, \binits{C.}}:
\batitle{Stochastic differential equations occurring in the estimation of continuous parameter stochastic processes}.
\bjtitle{Theory of Probability \& Its Applications}
\bvolume{14}(\bissue{4}),
\bfpage{567}--\blpage{594}
(\byear{1969})
\end{barticle}
\endbibitem

%%% 15
\bibitem[\protect\citeauthoryear{Duncan}{1967}]{duncan1967probability}
\begin{botherref}
\oauthor{\bsnm{Duncan}, \binits{T.E.}}:
Probability densities for diffusion processes with applications to nonlinear filtering theory and detection theory.
Stanford University
(1967)
\end{botherref}
\endbibitem

%%% 16
\bibitem[\protect\citeauthoryear{Mortensen}{1966}]{mortensen1966optimal}
\begin{botherref}
\oauthor{\bsnm{Mortensen}, \binits{R.E.}}:
Optimal control of continuous-time stochastic systems.
University of California, Berkeley
(1966)
\end{botherref}
\endbibitem

%%% 17
\bibitem[\protect\citeauthoryear{Zakai}{1969}]{zakai1969optimal}
\begin{barticle}
\bauthor{\bsnm{Zakai}, \binits{M.}}:
\batitle{On the optimal filtering of diffusion processes}.
\bjtitle{Zeitschrift f{\"u}r Wahrscheinlichkeitstheorie und verwandte Gebiete}
\bvolume{11}(\bissue{3}),
\bfpage{230}--\blpage{243}
(\byear{1969})
\end{barticle}
\endbibitem

%%% 18
\bibitem[\protect\citeauthoryear{Ahmed and Li}{1991}]{ahmed1991quadratic}
\begin{barticle}
\bauthor{\bsnm{Ahmed}, \binits{N.}},
\bauthor{\bsnm{Li}, \binits{P.}}:
\batitle{Quadratic regulator theory and linear filtering under system constraints}.
\bjtitle{IMA Journal of Mathematical Control and Information}
\bvolume{8}(\bissue{1}),
\bfpage{93}--\blpage{107}
(\byear{1991})
\end{barticle}
\endbibitem

%%% 19
\bibitem[\protect\citeauthoryear{Ahmed and Charalambous}{2002}]{ahmed2002filtering}
\begin{barticle}
\bauthor{\bsnm{Ahmed}, \binits{N.U.}},
\bauthor{\bsnm{Charalambous}, \binits{C.D.}}:
\batitle{Filtering for linear systems driven by fractional brownian motion}.
\bjtitle{SIAM Journal on Control and Optimization}
\bvolume{41}(\bissue{1}),
\bfpage{313}--\blpage{330}
(\byear{2002})
\end{barticle}
\endbibitem

%%% 20
\bibitem[\protect\citeauthoryear{Hu}{2005}]{hu2005integral}
\begin{botherref}
\oauthor{\bsnm{Hu}, \binits{Y.}}:
Integral transformations and anticipative calculus for fractional brownian motions.
American Mathematical Soc.
(2005)
\end{botherref}
\endbibitem

%%% 21
\bibitem[\protect\citeauthoryear{Nualart and Nualart}{2018}]{nualart2018introduction}
\begin{botherref}
\oauthor{\bsnm{Nualart}, \binits{D.}},
\oauthor{\bsnm{Nualart}, \binits{E.}}:
Introduction to malliavin calculus.
vol. 9.
Cambridge University Press
(2018)
\end{botherref}
\endbibitem

%%% 22
\bibitem[\protect\citeauthoryear{Han et~al.}{2013}]{han2013maximum}
\begin{barticle}
\bauthor{\bsnm{Han}, \binits{Y.}},
\bauthor{\bsnm{Hu}, \binits{Y.}},
\bauthor{\bsnm{Song}, \binits{J.}}:
\batitle{Maximum principle for general controlled systems driven by fractional brownian motions}.
\bjtitle{Applied Mathematics \& Optimization}
\bvolume{67}(\bissue{2}),
\bfpage{279}--\blpage{322}
(\byear{2013})
\end{barticle}
\endbibitem

%%% 23
\bibitem[\protect\citeauthoryear{Nunno et~al.}{2008}]{nunno2008malliavin}
\begin{botherref}
\oauthor{\bsnm{Nunno}, \binits{G.D.}},
\oauthor{\bsnm{{\O}ksendal}, \binits{B.}},
\oauthor{\bsnm{Proske}, \binits{F.}}:
Malliavin calculus for l{\'e}vy processes with applications to finance.
Springer
(2008)
\end{botherref}
\endbibitem

\end{thebibliography}
\end{document}